\documentclass[preprint,12pt]{elsarticle}
\usepackage{amsfonts}

 \usepackage{graphics}
\usepackage{graphicx}
 \usepackage{epsfig}

\usepackage{amssymb}
\usepackage{amsthm}

\newtheorem{thm}{Theorem}
\newtheorem{lem}[thm]{Lemma}
\newdefinition{rmk}{Remark}
\newproof{pf}{Proof}
\newdefinition{example}{Example}
\newdefinition{definition}{Definition}
\newdefinition{proposition}{Proposition}
\newdefinition{corollary}{Corollary}

\journal{arXiv}

\begin{document}

\begin{frontmatter}



\title{Improvements on the infinity norm bound for the inverse of Nekrasov matrices}


\author[rvt]{Chaoqian Li}
\author[rvt]{Hui Pei}
\author[rvt]{Aning Gao}
\author[rvt]{Yaotang Li\corref{cor1}}
\ead{liyaotang@ynu.edu.cn}

\cortext[cor1]{Corresponding author.}

\address[rvt]{School of Mathematics and Statistics, Yunnan
University, Kunming, Yunnan,  P. R. China 650091}
\begin{abstract}
We focus on the estimating problem of  the infinity norm of the
inverse of Nekrasov matrices, give new bounds which involve a
parameter, and then determine the optimal value of the parameter
such that the new bounds are better than those in L. Cvetkovi\'c et al. (2013) \cite{Cv3}.  Numerical examples are given
to illustrate the corresponding results.
\end{abstract}

\end{frontmatter}


\section{Introduction}
A matrix $A=(a_{ij})\in C^{n, n}$ is called an $H$-matrix if its
comparison matrix $ <A>=[m_{ij}]$ defined by
\[ <A>=[m_{ij}]\in  C^{n, n},m_{ij}=\left\{ \begin{array}{cc}
  |a_{ii}|,  &i=j \\
 -|a_{ij}|,  &i\neq j,
\end{array} \right.\]
is an $M$-matrix, i.e., $<A>^{-1}\geq 0$ \cite{Ba,Cv1,Cv2}.
$H$-matrices has a large number of applications. One special
interest problem among them is to find upper bounds of the infinity
norm of $H$-matrices, since it can be used for proving the
convergence of matrix splitting and matrix multisplitting iteration
methods for solving large sparse systems of linear equations, see
\cite{Ba,Cv3,Hu,Hu1}. Many researchers gave some well-known bounds.
In 1975, J.M. Varah\cite{Va} provided the follwoing upper bound for
strictly diagonally dominant (SDD) matrices as one most important
subclass of $H$-matrices. Here a matrix $A=[a_{ij}]\in C^{n, n}$ is
called SDD if for each $i\in N=\{1,2,\ldots,n\}$,
\[|a_{ii}|>r_i(A),\]
where $r_i(A)=\sum\limits_{j\neq i} |a_{ij}|$.

\begin{thm} \emph{\cite{Va}} \label{th1.1}
Let $A=[a_{ij}]\in C^{n, n}$ be SDD. Then
\[ ||A^{-1}||_\infty \leq \frac{1}{\min\limits_{i\in N} (|a_{ii}|-r_i(A))}.\]
\end{thm}

We call the bound in Theorem \ref{th1.1} the Varah's bound. As
Cvetkovi\'c  et al. \cite{Cv3} said, the Varah's bound works only
for SDD matrices, and even then it is not always good enough. Hence,
it can be useful to obtain new upper bounds for a wider class of
matrices which  sometimes works better in the SDD case. In
\cite{Cv3}, Cvetkovi\'c  et al. study the class of Nekrasov matrices
which contains SDD matrices and is a subclass of $H$-matrices, and
give the following bounds (see Theorem \ref{th1.2}).

\begin{definition} \cite{Cv2,Cv3}
A matrix $A=[a_{ij}]\in C^{n, n}$ is called a Nekrasov matrix if for
each $i\in N$,
\[|a_{ii}|>h_i(A),\]
where $h_1(A)=r_1(A)=\sum\limits_{j\neq 1} |a_{1j}|$ and
$h_i(A)=\sum\limits_{j=1}^{i-1}\frac{|a_{ij}|}{|a_{jj}|}h_j(A)+\sum\limits_{j=i+1}^{n}|a_{ij}|$,
$i=2,3,\ldots,n$.
\end{definition}

\begin{thm} \emph{\cite[Theorem 2]{Cv3}} \label{th1.2}
Let $A=[a_{ij}]\in C^{n, n}$ be a Nekrasov matrix. Then
\begin{equation} \label{eq1.1} ||A^{-1}||_\infty \leq \frac{\max\limits_{i\in
N}\frac{z_i(A)}{|a_{ii}|}} { 1-\max\limits_{i\in N}
\frac{h_i(A)}{|a_{ii}|}},\end{equation}
and
\begin{equation} \label{eq1.10} ||A^{-1}||_\infty \leq \frac{\max\limits_{i\in
N}z_i(A)} { \min\limits_{i\in N}(|a_{ii}|- h_i(A))},\end{equation}
where $z_1(A)=1$ and $z_i(A)=\sum\limits_{j=1}^{i-1}
\frac{|a_{ij}|}{|a_{jj}|}z_j(A)+1,$ $i=2,3\ldots,n$.
\end{thm}

Since an SDD matrix is a Nekrasov matrices \cite{Cv2,Li}, the bounds
(\ref{eq1.1}) and (\ref{eq1.10}) can be also applied to  SDD
matrices. However, the Varah's bound can not be used to estimate the
infinity norm of the inverse of Nekrasov matrices. Furthermore, when
we use both bounds to estimate the infinity norm of the inverse of
SDD matrices, the bound (\ref{eq1.1}) or (\ref{eq1.10}) works better
than the Varah's bound in some cases (for details, see \cite{Cv3}).

In this paper, we also focus on the estimating problem of the
infinity norm of the inverse of Nekrasov matrices, and give new
bounds which involve a parameter $\mu$ based on the bounds in
Theorem \ref{th1.2}, and then determine the optimal value of $\mu$
such that the new bounds are better than those in Theorem
\ref{th1.2} (Theorem 2 in \cite {Cv3}). Numerical examples are given
to illustrate the corresponding results.

\section {New bounds for the infinity norm of the inverse of Nekrasov matrices}
First, some lemmas and notations are listed. Given a matrix
$A=[a_{ij}]$, by $A=D-L- U$ we denote the standard splitting of $A$
into its diagonal $(D)$, strictly lower $(-L)$ and strictly upper
$(-U)$ triangular parts. And by $[A]_{ij}$ we denote the
$(i,j)$-entry of $A$, that is, $[A]_{ij}=a_{ij}$.

\begin{lem}\label{le2.0} \emph{\cite{Be}}
Let $A=[a_{ij}]\in C^{n,n}$ be a nonsingular $H$-matrix. Then
\[ |A^{-1}|\leq ~<A>^{-1}.\]
\end{lem}

\begin{lem}\emph{\cite{Ro}}\label{le 2.1}
Given any matrix $A=[a_{ij}]\in C^{n,n}$, $n\geq 2$, with
$a_{ii}\neq 0$ for all $i\in N$, then \[h_i(A)=
|a_{ii}|\left[(|D|-|L|)^{-1}|U|e\right]_i,\] where $e\in C^{n,n}$ is
the vector with all components equal to 1.
\end{lem}

\begin{lem}\label{le 2.2} \emph{\cite{Sz}}
A matrix $A=[a_{ij}]\in C^{n,n}$, $n\geq 2$ is  a Nekrasov matrix if
and only if  \[(|D|-|L|)^{-1}|U|e<e,\] i.e., if and only if
$E-(|D|-|L|)^{-1}|U|$ is an SDD matrix, where $E$ is the identity
matrix.
\end{lem}

Let\[C=E-(|D|-|L|)^{-1}|U|=[c_{ij}]\] and
\[B=|D|C= |D|-|D|(|D|-|L|)^{-1}|U|=[b_{ij}]\]
and
 Then from Lemma \ref{le 2.2}, $B$
and $C$ are SDD when $A$ is a Nekrasov matrix. Note that $c_{11}=1$,
$c_{k1}=0$, $k=2,3,\ldots,n$, and
$c_{1k}=-\frac{|a_{1k}|}{|a_{11}|}$, $k=2,3,\ldots,n$, and that
$b_{11}=|a_{11}|$, $b_{k1}=0$, $k=2,3,\ldots,n$, and
$b_{1k}=-|a_{1k}|$, $k=2,3,\ldots,n$, which lead to the following
lemma.

\begin{lem}\label{le 2.3}
Let $A=[a_{ij}] \in C^{n,n}$ be a Nekrasov matrix,
\begin{equation} \label{eq2.0}
C(\mu)=CD(\mu)=\left(E-(|D|-|L|)^{-1}|U|\right)D(\mu),\end{equation}
and
\begin{equation}
\label{eq2.00}
B(\mu)=BD(\mu)=\left(|D|-|D|(|D|-|L|)^{-1}|U|\right)D(\mu),\end{equation}
where $D(\mu)=diag(\mu,1,\cdots,1)$ and $\mu>
\frac{r_1(A)}{|a_{11}|}$. Then $C(\mu)$ and  $B(\mu)$  are SDD,
\begin{equation} \label{eq2.1} ||C(\mu)^{-1}||_\infty \leq
\max\left\{\frac{1}{\mu - \frac{h_1(A)}{|a_{11}|}}, \frac{1}{
1-\max\limits_{i\neq
1}\frac{h_i(A)}{|a_{ii}|}}\right\},\end{equation} and
\begin{equation} \label{eq2.10} ||B(\mu)^{-1}||_\infty \leq
\frac{1}{\min\left\{\mu|a_{11}| - h_1(A), \min\limits_{i\neq
1}(|a_{ii}|-h_i(A))\right\}}. \end{equation}
 \end{lem}

\begin{proof} We first prove (\ref{eq2.1}) holds. It is not difficult from (\ref{eq2.0}) to see that
$[C(\mu)]_{k1}=\mu c_{k1}$ for all $k\in N$ and
$[C(\mu)]_{kj}=c_{kj}$ for all $k\in N$ and $j\neq 1$. Hence
\[[C(\mu)]_{11}=\mu,~ r_1(C(\mu))=r_1(C)=\frac{r_1(A)}{|a_{11}|}\]
and for $i=2,\ldots,n$, \[ [C(\mu)]_{ii}=c_{ii},~
r_i(C(\mu))=r_i(C).\] From $C$ is SDD and $\mu>
\frac{r_1(A)}{|a_{11}|}$, we have that $C(\mu) $ is SDD.

Moreover, by applying the Varah's bound to estimate the infinity
norm of its inverse matrix, we can obtain
\[|| C(\mu)^{-1}||_\infty \leq \max\limits_{i\in N}
\frac{1}{|[C(\mu)]_{ii}|- r_i(C(\mu) )}=\max\left\{\frac{1}{\mu -
r_1(C)}, \max\limits_{i\neq 1} \frac{1}{|c_{ii}|- r_i(C)}\right\}.\]
Note that $C=E-(|D|-|L|)^{-1}|U|=[c_{ij}]$ and all diagonal entries
of matrix $(|D|-|L|)^{-1}|U|$ are less than 1. Then we have that for
$i\in N, ~i\neq 1, $
\[|c_{ii}|= 1-\left[(|D|-|L|)^{-1}|U| \right]_{ii}\]
and that for each $i\in N$, \[ r_i(C)=\sum\limits_{k\neq
i}\left[(|D|-|L|)^{-1}|U| \right]_{ik}.\] These lead to that (also
see the proof of Theorem 2 in \cite{Cv3}) for $i\in N,~i\neq 1$,
\[ |c_{ii}|-r_i(C)=1-\sum\limits_{k\in N} \left[(|D|-|L|)^{-1}|U| \right]_{ik}=
1-\left[(|D|-|L|)^{-1}|U|e \right]_{i}=1-\frac{h_i(A)}{|a_{ii}|}.\]
Since $r_1(C)=\frac{r_1(A)}{|a_{11}|}=\frac{h_1(A)}{|a_{11}|}$, we
have
\begin{eqnarray*}|| C(\mu)^{-1}||_\infty&\leq & \max\left\{\frac{1}{\mu - r_1(C)}, \max\limits_{i\neq 1}
\frac{1}{|c_{ii}|- r_i(C)}\right\}\\
&=& \max\left\{\frac{1}{\mu - \frac{h_1(A)}{|a_{11}|}},
\max\limits_{i\neq 1}\frac{1}{
1-\frac{h_i(A)}{|a_{ii}|}}\right\}\\
&=&\max\left\{\frac{1}{\mu - \frac{h_1(A)}{|a_{11}|}}, \frac{1}{
1-\max\limits_{i\neq 1}\frac{h_i(A)}{|a_{ii}|}}\right\}.
\end{eqnarray*}

We prove easily  that (\ref{eq2.10}) holds in an analogous way. The
proof is completed.
\end{proof}

The main result of this paper is the following theorem:

\begin{thm}\label{th 2.1}
Let $A=[a_{ij}] \in C^{n,n}$ be a Nekrasov matrix. Then for
$\mu>\frac{r_1(A)}{|a_{11}|}$,
\begin{equation} \label{m1}||A^{-1}||_\infty \leq \max\{ \mu,1\}
\max\limits_{i\in N} \frac{z_i(A)}{|a_{ii}|}\max\left\{\frac{1}{\mu
- \frac{h_1(A)}{|a_{11}|}}, \frac{1}{ 1-\max\limits_{i\neq
1}\frac{h_i(A)}{|a_{ii}|}}\right\},\end{equation} and
\begin{equation}\label{m10} ||A^{-1}||_\infty \leq
 \frac{ \max\{ \mu,1\} \max\limits_{i\in n}
z_i(A) }{\min\left\{\mu|a_{11}| - h_1(A), \min\limits_{i\neq
1}(|a_{ii}|-h_i(A))\right\}}.\end{equation}
\end{thm}

\begin{proof}We only prove that (\ref{m1}) holds, and in an analogous way, (\ref{m10}) is proved easily.
 Let $C(\mu)=CD(\mu)=\left(E-(|D|-|L|)^{-1}|U|\right)D(\mu)$,
where $D(\mu)=diag(\mu,1,\cdots,1)$. From (\ref{eq2.0}), we have
\[C(\mu)=\left(E-(|D|-|L|)^{-1}|U|\right)D(\mu)=(|D|-|L|)^{-1}<A>D(\mu),\]
which implies that
\[<A>=(|D|-|L|)C(\mu) D(\mu)^{-1}.\]
Furthermore, since a Nekrasov matrix is an $H$-matrix, we have from
Lemma \ref{le2.0},
\begin{equation}\label{eq2.2}|| A^{-1}||_\infty\leq || <A>^{-1}||_\infty \leq ||D(\mu) ||_\infty||C(\mu)^{-1} ||_\infty
|| (|D|-|L|)^{-1}||_\infty.\end{equation} Note that $|D|-|L| $ is an
$M$-matrix, and then similar to the proof of Theorem 2 in
\cite{Cv3}, we can easily obtain
\begin{equation}\label{eq2.3} ||(|D|-|L|)^{-1}||_\infty= ||y ||_\infty=\max\limits_{i\in n}
\frac{z_i(A)}{|a_{ii}|}, \end{equation} where $y=
(|D|-|L|)^{-1}e=[y_1,y_2,\ldots,y_n]^T$ and $z_i(A)=|a_{ii}|y_i$,
i.e.,
\[ z_1(A)=1, ~and ~z_i(A)=\sum\limits_{j=1}^{i-1} \frac{|a_{ij}|}{|a_{jj}|}z_j(A)+1, i=2,\ldots,n.\]
From (\ref{eq2.1}), (\ref{eq2.2}), (\ref{eq2.3}) and the fact that
$||D(\mu) ||_\infty= \max\{ \mu,1\}$, we have
\[||A^{-1}||_\infty \leq \max\{ \mu,1\}
\max\limits_{i\in N} \frac{z_i(A)}{|a_{ii}|}\max\left\{\frac{1}{\mu
- \frac{h_1(A)}{|a_{11}|}}, \frac{1}{ 1-\max\limits_{i\neq
1}\frac{h_i(A)}{|a_{ii}|}}\right\}.\] The conclusions follows.
\end{proof}

\begin{example} \label{ex1} Consider the Nekrasov matrix $A_1$ in \cite{Cv3}, where \[A_1 =
\left[ \begin{array}{cccc}
  -7  &1   &-0.2 &2 \\
 7   &88    &2 &-3   \\
 2   &0.5    &13 &-2\\
 0.5 & 3.0 &1  & 6
\end{array} \right].\]
By computation, $h_1(A)= 3.2000,~h_2(A)= 8.2000,~h_3(A)=
2.9609,~h_4(A)= 0.7359,$ $ z_1(A)=1,$ $z_2(A)=2,$ $z_3(A)=1.2971$
and $z_4(A)=1.2394$. By Theorem \ref{th1.2} (Theorem 2 in
\cite{Cv3}), we have
\[||A_1^{-1}||_\infty \leq  0.3805, ~(The~ bound ~(\ref{eq1.1}) ~of~Theorem ~\ref{th1.2})\]
and \[||A_1^{-1}||_\infty \leq 0.5263. ~(The~ bound ~(\ref{eq1.10})
~of~Theorem ~\ref{th1.2})\] By the bound (\ref{m1}) of Theorem
\ref{th 2.1}, we have
\begin{eqnarray*}||A_1^{-1}||_\infty \leq & 4.8198 & ~(Taking ~\mu=0.5),\\
||A_1^{-1}||_\infty \leq &  0.6025 & ~(Taking
~\mu=0.8),\\||A_1^{-1}||_\infty \leq &  0.3535 & ~(Taking ~\mu=1.1),
\\||A_1^{-1}||_\infty \leq &  0.3745 & ~(Taking
~\mu=1.4),\\||A_1^{-1}||_\infty \leq &  0.4547 & ~(Taking
~\mu=1.7),\end{eqnarray*}
 and by the bound (\ref{m10}) of Theorem
\ref{th 2.1}, we have
\begin{eqnarray*}||A_1^{-1}||_\infty \leq &  2.0000 & ~(Taking ~\mu=0.6),\\
||A_1^{-1}||_\infty \leq &  0.6452 & ~(Taking
~\mu=0.9),\\||A_1^{-1}||_\infty \leq &   0.4615 & ~(Taking
~\mu=1.2),
\\||A_1^{-1}||_\infty \leq &   0.5699 & ~(Taking
~\mu=1.5),\\||A_1^{-1}||_\infty \leq &  0.6839 & ~(Taking
~\mu=1.8).\end{eqnarray*} In fact, $||A_1^{-1}||_\infty= 0.1921,$

\begin{figure}[bpt]
\centering
\begin{minipage}[b]{0.5\textwidth}
\centering
\includegraphics[width=2.4in]{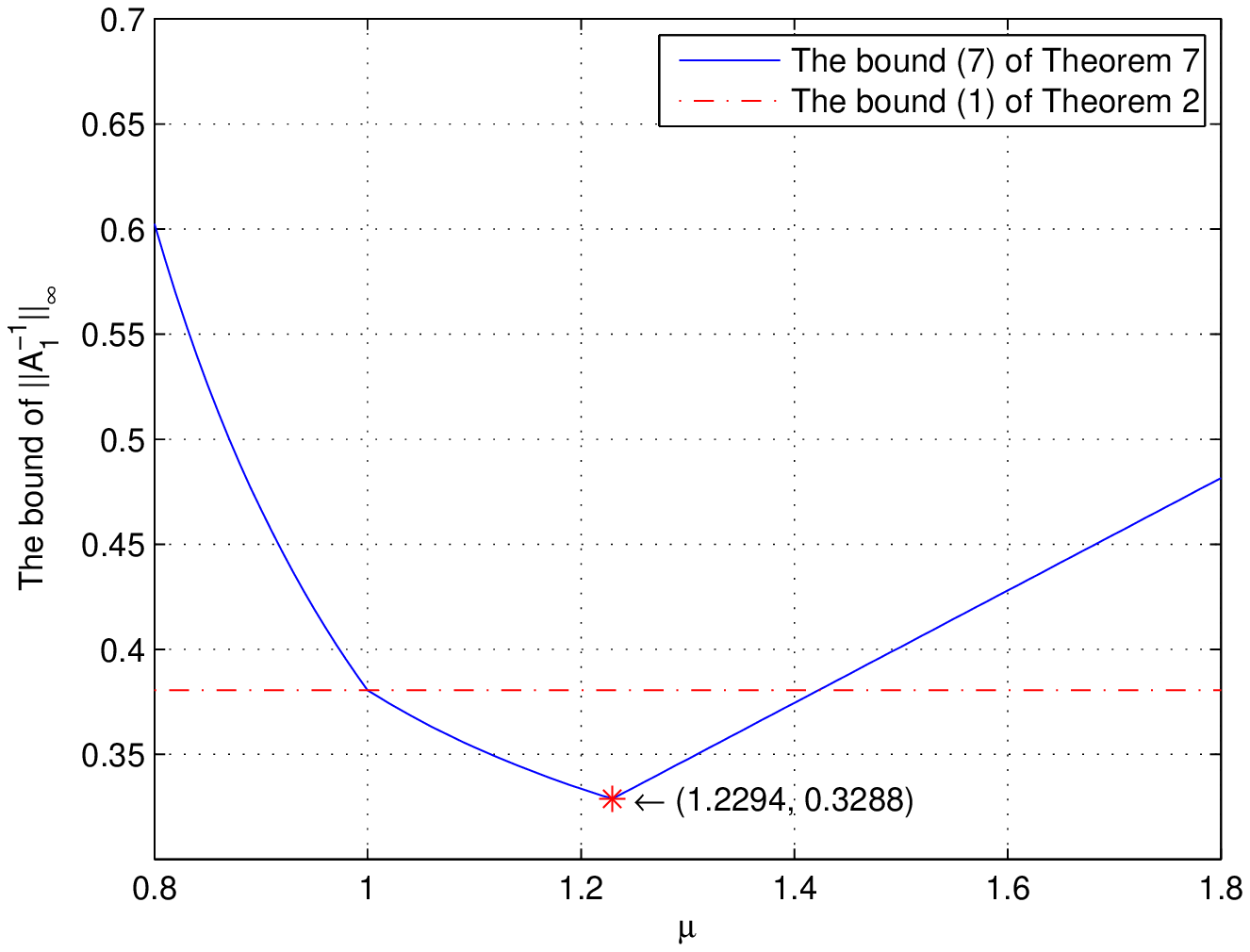}
\caption{The bounds (\ref{eq1.1}) and (\ref{m1})}
\end{minipage}%
\begin{minipage}[b]{0.5\textwidth}
\centering
\includegraphics[width=2.4in]{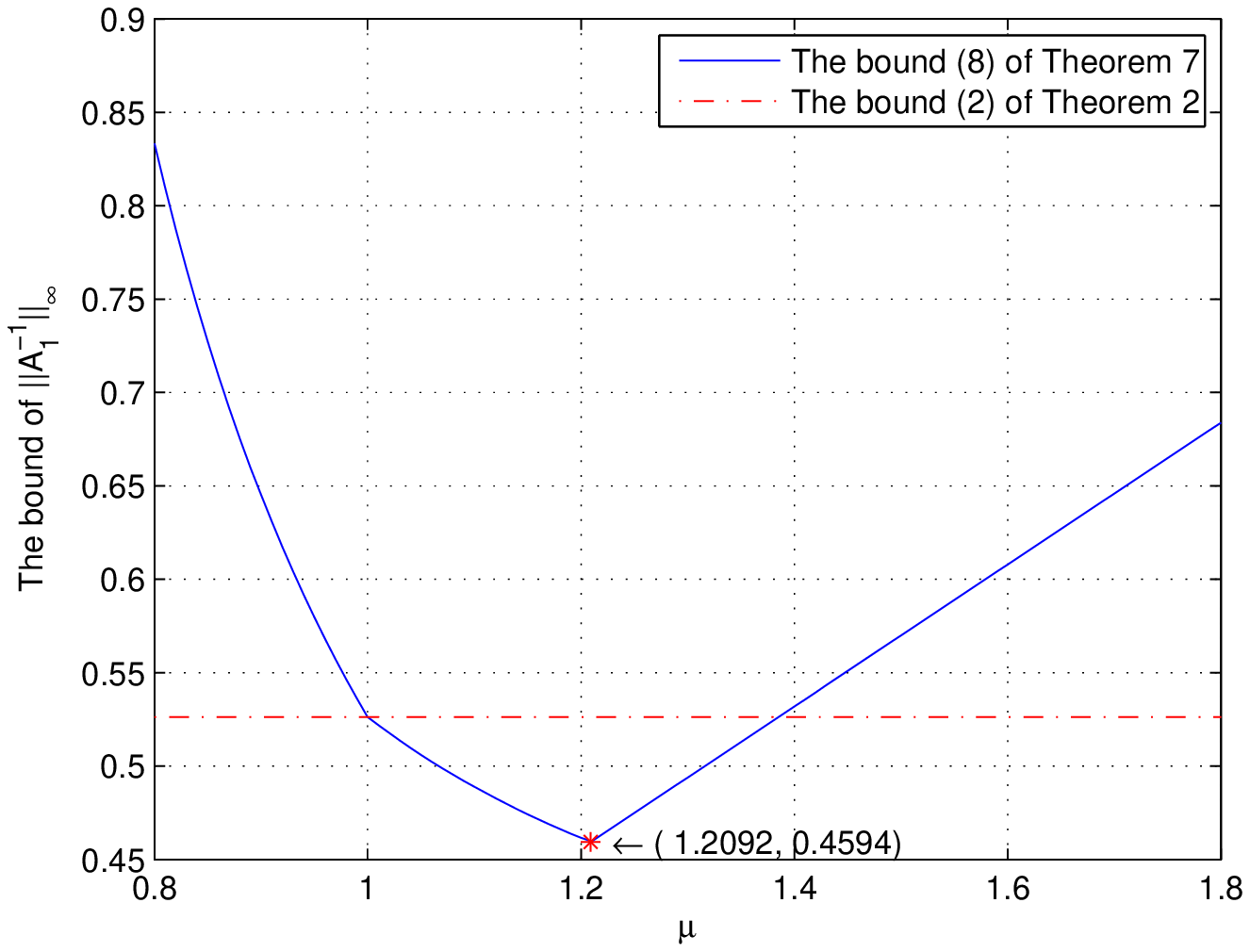}
\caption{The bounds (\ref{eq1.10}) and (\ref{m10})}
\end{minipage}
\end{figure}

\end{example}

\begin{rmk} Example \ref{ex1} shows that by choosing the value of
$\mu$, the bound (\ref{m1}) ((\ref{m10}), resp.) of Theorem \ref{th
2.1} is better than the bound (\ref{eq1.1}) ((\ref{eq1.10}), resp.)
of Theorem \ref{th1.2} in some cases. We further observe the bounds
in Theorem \ref{th 2.1} by Figures 1 and 2, and find that there is
an interval such that for any $\mu$ in this interval, the bound
(\ref{m1}) ((\ref{m10}), resp.) of Theorem \ref{th 2.1} for the
matrix $A_1$ is always smaller than the bound (\ref{eq1.1})
((\ref{eq1.10}), resp.) of Theorem \ref{th1.2}. An interesting
problem arises: whether there is an interval of $\mu$ such that the
bound (\ref{m1}) ((\ref{m10}), resp.) of Theorem \ref{th 2.1} for
any Nekrasov matrix is smaller than the bound (\ref{eq1.1})
((\ref{eq1.10}), resp.) of Theorem \ref{th1.2}? In the following
section, we will study this problem.
\end{rmk}

\section{The choice of $\mu$} In this section, we determine the value of $\mu$ such that our
bounds for $||A^{-1}||_\infty$ are less or equal to those of
\cite{Cv3}.
\subsection{\textbf{the optimal value of  $\mu$ for the bound (\ref{m1})}}
First, we consider the Nekrasov matrix $A=[a_{ij}] \in C^{n,n}$ with
\begin{equation} \label{eq2.4}\frac{h_1(A)}{|a_{11}|}>
\max\limits_{i\neq 1} \frac{h_i(A)}{|a_{ii}|},\end{equation} and
give the following lemma.

\begin{lem}\label{le 2.4}
Let $A=[a_{ij}] \in C^{n,n}$ be a Nekrasov matrix with \[
\frac{h_1(A)}{|a_{11}|}> \max\limits_{i\neq 1}
\frac{h_i(A)}{|a_{ii}|}.\] Then
\begin{equation} \label{eq2.6} 1<1+\frac{h_1(A)}{|a_{11}|}-\max\limits_{i\neq 1}
\frac{h_i(A)}{|a_{ii}|}<\frac{1-\max\limits_{i\neq
1}\frac{h_i(A)}{|a_{ii}|}}{1 - \frac{h_1(A)}{|a_{11}|}
}.\end{equation}
\end{lem}

\begin{proof}
Obviously,  the first Inequality in (\ref{eq2.6}) holds. We only
prove that the second holds. From Inequality (\ref{eq2.4}), we have
that
\[ \frac{h_1(A)}{|a_{11}|}\max\limits_{i\neq 1}
\frac{h_i(A)}{|a_{ii}|}- \left(\frac{r_1(A)}{|a_{11}|}\right)^2<0.\]
Equivalently,
\[1-\max\limits_{i\neq 1}
\frac{h_i(A)}{|a_{ii}|}+\frac{h_1(A)}{|a_{11}|}-\frac{h_1(A)}
{|a_{11}|}+\frac{h_1(A)}{|a_{11}|}\max\limits_{i\neq 1}
\frac{h_i(A)}{|a_{ii}|}-
\left(\frac{h_1(A)}{|a_{11}|}\right)^2<1-\max\limits_{i\neq 1}
\frac{h_i(A)}{|a_{ii}|},\] i.e.,
 \[ \left(1-\max\limits_{i\neq 1}
\frac{h_i(A)}{|a_{ii}|}+\frac{h_1(A)}{|a_{11}|} \right) \left(
1-\frac{h_1(A)}{|a_{11}|}\right)< 1-\max\limits_{i\neq 1}
\frac{h_i(A)}{|a_{ii}|}.\] Note that $1-\frac{h_1(A)}{|a_{11}|}>0 $,
then
\[ 1-\max\limits_{i\neq 1}
\frac{h_i(A)}{|a_{ii}|}+\frac{h_1(A)}{|a_{11}|}<\frac{1-\max\limits_{i\neq
1}\frac{h_i(A)}{|a_{ii}|}}{1 - \frac{h_1(A)}{|a_{11}|} }.\] The
conclusion follows. \end{proof}

We now give an interval of $\mu$ such that the bound (\ref{m1}) of
Theorem \ref{th 2.1} is less than the bound (\ref{eq1.1}) of Theorem
\ref{th1.2}.

\begin{lem}\label{le 2.5}
Let $A=[a_{ij}] \in C^{n,n}$ be a Nekrasov matrix with
\[\frac{h_1(A)}{|a_{11}|}> \max\limits_{i\neq 1}
\frac{h_i(A)}{|a_{ii}|}.\]Then for each $\mu \in \left(1,
\frac{1-\max\limits_{i\neq 1}\frac{h_i(A)}{|a_{ii}|}}{1 -
\frac{h_1(A)}{|a_{11}|} }\right)$,
\begin{eqnarray} \label{eq2.5}||A^{-1}||_\infty &\leq& \max\{ \mu,1\}
\max\limits_{i\in N} \frac{z_i(A)}{|a_{ii}|}\max\left\{\frac{1}{\mu
- \frac{h_1(A)}{|a_{11}|}}, \frac{1}{ 1-\max\limits_{i\neq
1}\frac{h_i(A)}{|a_{ii}|}}\right\}\nonumber
\\& <&\frac{\max\limits_{i\in
N}\frac{z_i(A)}{|a_{ii}|}}{ 1-\max\limits_{i\in
N}\frac{h_i(A)}{|a_{ii}|}}.\nonumber\end{eqnarray}
\end{lem}

\begin{proof} From Lemma \ref{le 2.4}, we have
\[\mu \in\left(1,1+\frac{h_1(A)}{|a_{11}|}-\max\limits_{i\neq 1}
\frac{h_i(A)}{|a_{ii}|}\right]\bigcup
\left[1+\frac{h_1(A)}{|a_{11}|}-\max\limits_{i\neq 1}
\frac{h_i(A)}{|a_{ii}|}, \frac{1-\max\limits_{i\neq
1}\frac{h_i(A)}{|a_{ii}|}}{1 - \frac{h_1(A)}{|a_{11}|} }\right).\]
and $ \max\{ \mu,1\}=\mu$.

(I) For $ \mu \in
\left(1,1+\frac{h_1(A)}{|a_{11}|}-\max\limits_{i\neq 1}
\frac{h_i(A)}{|a_{ii}|}\right],$ then \[ \mu
-\frac{h_1(A)}{|a_{11}|}\leq 1-\max\limits_{i\neq 1}
\frac{h_i(A)}{|a_{ii}|},\] that is, \[ \frac{1}{\mu
-\frac{h_1(A)}{|a_{11}|}}\geq \frac{1}{1-\max\limits_{i\neq 1}
\frac{h_i(A)}{|a_{ii}|} }.\] Therefore, \[ \max\{ \mu,1\}
\max\left\{\frac{1}{\mu - \frac{h_1(A)}{|a_{11}|}}, \frac{1}{
1-\max\limits_{i\neq 1}\frac{h_i(A)}{|a_{ii}|}}\right\}= \frac{\mu}{
\mu - \frac{h_1(A)}{|a_{11}|}}.\] Consider the function
$f(x)=\frac{x}{x-\frac{h_1(A)}{|a_{11}|}}, ~x\in
\left[1,1+\frac{h_1(A)}{|a_{11}|}-\max\limits_{i\neq 1}
\frac{h_i(A)}{|a_{ii}|}\right]$. It is easy from
$\frac{h_1(A)}{|a_{11}|}<1$ to prove that $f(x)$ is a monotonically
decreasing function of $x$. Hence, for any $ \mu \in
\left(1,1+\frac{h_1(A)}{|a_{11}|}-\max\limits_{i\neq 1}
\frac{h_i(A)}{|a_{ii}|}\right]$,
\[f(\mu)< f(1),\]
i.e., \[ \frac{\mu}{ \mu - \frac{h_1(A)}{|a_{11}|}}< \frac{1}{ 1 -
\frac{h_1(A)}{|a_{11}|}}=\frac{1}{1-\max\limits_{i\in
N}\frac{h_i(A)}{|a_{ii}|}},\] which implies that
\[\frac{\mu\max\limits_{i\in N}\frac{z_i(A)}{|a_{ii}|}}{ \mu - \frac{h_1(A)}{|a_{11}|}}<
\frac{\max\limits_{i\in
N}\frac{z_i(A)}{|a_{ii}|}}{1-\max\limits_{i\in
N}\frac{h_i(A)}{|a_{ii}|}}. \] Hence,
 \[\max\{ \mu,1\}
\max\limits_{i\in N} \frac{z_i(A)}{|a_{ii}|}\max\left\{\frac{1}{\mu
- \frac{h_1(A)}{|a_{11}|}}, \frac{1}{ 1-\max\limits_{i\neq
1}\frac{h_i(A)}{|a_{ii}|}}\right\}\nonumber <\frac{\max\limits_{i\in
N}\frac{z_i(A)}{|a_{ii}|}}{ 1-\max\limits_{i\in
N}\frac{h_i(A)}{|a_{ii}|}}.\]

(II) For $ \mu \in
\left[1+\frac{h_1(A)}{|a_{11}|}-\max\limits_{i\neq 1}
\frac{h_i(A)}{|a_{ii}|}, \frac{1-\max\limits_{i\neq
1}\frac{h_i(A)}{|a_{ii}|}}{1 - \frac{h_1(A)}{|a_{11}|} }\right),$
then
\[ \mu
-\frac{h_1(A)}{|a_{11}|}\geq 1-\max\limits_{i\neq 1}
\frac{h_i(A)}{|a_{ii}|},\] that is, \[ \frac{1}{\mu
-\frac{h_1(A)}{|a_{11}|}}\leq \frac{1}{1-\max\limits_{i\neq 1}
\frac{h_i(A)}{|a_{ii}|} }.\] Therefore, \[ \max\{ \mu,1\}
\max\left\{\frac{1}{\mu - \frac{h_1(A)}{|a_{11}|}}, \frac{1}{
1-\max\limits_{i\neq 1}\frac{h_i(A)}{|a_{ii}|}}\right\}=
\frac{\mu}{1-\max\limits_{i\neq 1} \frac{h_i(A)}{|a_{ii}|} }.\]
Consider the function $g(x)=\frac{x}{1-\max\limits_{i\neq 1}
\frac{h_i(A)}{|a_{ii}|}}, ~x\in
\left[1+\frac{h_1(A)}{|a_{11}|}-\max\limits_{i\neq 1}
\frac{h_i(A)}{|a_{ii}|}, \frac{1-\max\limits_{i\neq
1}\frac{h_i(A)}{|a_{ii}|}}{1 - \frac{h_1(A)}{|a_{11}|} }\right]$.
Obviously, $g(x)$ is a monotonically increasing function of $x$.
Hence, for any $ \mu \in
\left[1+\frac{h_1(A)}{|a_{11}|}-\max\limits_{i\neq 1}
\frac{h_i(A)}{|a_{ii}|}, \frac{1-\max\limits_{i\neq
1}\frac{h_i(A)}{|a_{ii}|}}{1 - \frac{h_1(A)}{|a_{11}|} } \right)$,
\[g(\mu)< g\left(  \frac{1-\max\limits_{i\neq
1}\frac{h_i(A)}{|a_{ii}|}}{1 - \frac{h_1(A)}{|a_{11}|} }\right),\]
that is,
\[ \frac{\mu}{1-\max\limits_{i\neq 1} \frac{h_i(A)}{|a_{ii}|}}
<\frac{1}{1- \frac{h_1(A)}{|a_{11}|}} =\frac{1}{1-\max\limits_{i\in
N}\frac{h_i(A)}{|a_{ii}|}}, \] which implies that
\[ \frac{\mu\max\limits_{i\in N}\frac{z_i(A)}{|a_{ii}|}}{ 1-\max\limits_{i\neq
1}\frac{h_i(A)}{|a_{ii}|}} <\frac{\max\limits_{i\in
N}\frac{z_i(A)}{|a_{ii}|}}{ 1-\max\limits_{i\in
N}\frac{h_i(A)}{|a_{ii}|}}.\]Hence,
 \[\max\{ \mu,1\}
\max\limits_{i\in N} \frac{z_i(A)}{|a_{ii}|}\max\left\{\frac{1}{\mu
- \frac{h_1(A)}{|a_{11}|}}, \frac{1}{ 1-\max\limits_{i\neq
1}\frac{h_i(A)}{|a_{ii}|}}\right\}\nonumber <\frac{\max\limits_{i\in
N}\frac{z_i(A)}{|a_{ii}|}}{ 1-\max\limits_{i\in
N}\frac{h_i(A)}{|a_{ii}|}}.\] The conclusion follows from (I) and
(II). \end{proof}

Lemma \ref{le 2.5} provides an interval of $\mu$ such that the bound
(\ref{m1}) in Theorem \ref{th 2.1} is better than the bound
(\ref{eq1.1}) in Theorem \ref{th1.2}. Moreover, we can determine the
optimal value of $\mu$ by the following theorem.

\begin{thm}\label{th2.2}
Let $A=[a_{ij}] \in C^{n,n}$ be a Nekrasov matrix with
\[\frac{h_1(A)}{|a_{11}|}> \max\limits_{i\neq 1}
\frac{h_i(A)}{|a_{ii}|}.\] Then
\begin{eqnarray} \label{eq2.7} &\min\left\{ \max\{ \mu,1\}
\max\left\{\frac{1}{\mu - \frac{h_1(A)}{|a_{11}|}}, \frac{1}{
1-\max\limits_{i\neq 1}\frac{h_i(A)}{|a_{ii}|}}\right\}: \mu \in
\left(1, \frac{1-\max\limits_{i\neq 1}\frac{h_i(A)}{|a_{ii}|}}{1 -
\frac{h_1(A)}{|a_{11}|} }\right) \right\}\nonumber&
\\&= \frac{1+\frac{h_1(A)}{|a_{11}|}-\max\limits_{i\neq 1}
\frac{h_i(A)}{|a_{ii}|}}{ 1-\max\limits_{i\neq
1}\frac{h_i(A)}{|a_{ii}|}}.&\end{eqnarray} Furthermore,
\begin{equation} \label{eq2.8}||A^{-1}||_\infty \leq \frac{\max\limits_{i\in N}
\frac{z_i(A)}{|a_{ii}|}\left(1+\frac{h_1(A)}{|a_{11}|}-\max\limits_{i\neq
1} \frac{h_i(A)}{|a_{ii}|}\right)}{ 1-\max\limits_{i\neq
1}\frac{h_i(A)}{|a_{ii}|}}  <\frac{\max\limits_{i\in
N}\frac{z_i(A)}{|a_{ii}|}}{ 1-\max\limits_{i\in
N}\frac{h_i(A)}{|a_{ii}|}}.\end{equation}
\end{thm}

\begin{proof}
From the proof of Lemma \ref{le 2.5}, we have that
\[f(x)=\frac{x}{x-\frac{h_1(A)}{|a_{11}|}}, ~x\in
\left[1,1+\frac{h_1(A)}{|a_{11}|}-\max\limits_{i\neq 1}
\frac{h_i(A)}{|a_{ii}|}\right]\] is decreasing, and that
 \[g(x)=\frac{x}{1-\max\limits_{i\neq 1}
\frac{h_i(A)}{|a_{ii}|}}, ~x\in
\left[1+\frac{h_1(A)}{|a_{11}|}-\max\limits_{i\neq 1}
\frac{h_i(A)}{|a_{ii}|}, \frac{1-\max\limits_{i\neq
1}\frac{h_i(A)}{|a_{ii}|}}{1 - \frac{h_1(A)}{|a_{11}|} }\right]\] is
increasing. Therefore, the minimum of $ f(x)$, which is equal to
that of $g(x)$, is
\[ f\left(1+\frac{h_1(A)}{|a_{11}|}-\max\limits_{i\neq 1}
\frac{h_i(A)}{|a_{ii}|}\right)=g\left(1+\frac{h_1(A)}{|a_{11}|}-\max\limits_{i\neq
1}
\frac{h_i(A)}{|a_{ii}|}\right)=\frac{1+\frac{h_1(A)}{|a_{11}|}-\max\limits_{i\neq
1} \frac{h_i(A)}{|a_{ii}|}}{ 1-\max\limits_{i\neq
1}\frac{h_i(A)}{|a_{ii}|}},\] which implies that (\ref{eq2.7})
holds. Again by Lemma \ref{le 2.5}, (\ref{eq2.8}) follows easily.
\end{proof}

\begin{rmk} Theorem \ref{th2.2} provides a method to determine the optimal value of $\mu$
for a Nekrasov matrix $A=[a_{ij}] \in C^{n,n}$ with
\[\frac{h_1(A)}{|a_{11}|}> \max\limits_{i\neq 1}
\frac{h_i(A)}{|a_{ii}|}.\] Also consider the matrix $A_1$. By
computation, we get
\[ \frac{h_1(A_1)}{|a_{11}|}= 0.4571> 0.2278=\max\limits_{i\neq
1}\frac{h_{i}(A_1)}{|a_{ii}|}.\] Hence, by Theorem \ref{th2.2}, we
can obtain that the bound (\ref{m1}) in Theorem \ref{th 2.1} reaches
its minimum
\[ \frac{\max\limits_{i\in N}
\frac{z_i(A_1)}{|a_{ii}|}\left(1+\frac{h_1(A_1)}{|a_{11}|}-\max\limits_{i\neq
1} \frac{h_i(A_1)}{|a_{ii}|}\right)}{ 1-\max\limits_{i\neq
1}\frac{h_i(A_1)}{|a_{ii}|}} = 0.3288\] at
$\mu=1+\frac{h_1(A_1)}{|a_{11}|}-\max\limits_{i\neq 1}
\frac{h_i(A_1)}{|a_{ii}|}=1.2294$ (also see Figure 1).
\end{rmk}

Next, we study the bound in Theorem \ref{th 2.1} for the Nekrasov
matrix $A=[a_{ij}] \in C^{n,n}$ with
\[\frac{h_1(A)}{|a_{11}|}\leq \max\limits_{i\neq 1}
\frac{h_i(A)}{|a_{ii}|}.\]

\begin{thm} \label{th2.3}
Let $A=[a_{ij}] \in C^{n,n}$ be a Nekrasov matrix with
\[\frac{h_1(A)}{|a_{11}|}\leq \max\limits_{i\neq 1}
\frac{h_i(A)}{|a_{ii}|}.\] Then we can take $\mu
=1+\frac{h_1(A)}{|a_{11}|}-\max\limits_{i\neq 1}
\frac{h_i(A)}{|a_{ii}|}$ such that
\begin{eqnarray} ||A^{-1}||_\infty &\leq& \max\{ \mu,1\}
\max\limits_{i\in N} \frac{z_i(A)}{|a_{ii}|}\max\left\{\frac{1}{\mu
- \frac{h_1(A)}{|a_{11}|}}, \frac{1}{ 1-\max\limits_{i\neq
1}\frac{h_i(A)}{|a_{ii}|}}\right\}\nonumber
\\& =&\frac{\max\limits_{i\in
N}\frac{z_i(A)}{|a_{ii}|}}{ 1-\max\limits_{i\in
N}\frac{h_i(A)}{|a_{ii}|}}.\nonumber\end{eqnarray}
\end{thm}

\begin{proof} Since $\frac{h_1(A)}{|a_{11}|}\leq \max\limits_{i\neq 1}
\frac{h_i(A)}{|a_{ii}|}$, we have  $ \mu
=1+\frac{h_1(A)}{|a_{11}|}-\max\limits_{i\neq 1}
\frac{h_i(A)}{|a_{ii}|}\leq 1$,  $ \max\{ \mu,1\}=1$ and
\[\max\left\{\frac{1}{\mu -
\frac{h_1(A)}{|a_{11}|}}, \frac{1}{ 1-\max\limits_{i\neq
1}\frac{h_i(A)}{|a_{ii}|}}\right\}=\frac{1}{ 1-\max\limits_{i\neq
1}\frac{h_i(A)}{|a_{ii}|}}=\frac{1}{ 1-\max\limits_{i\in
N}\frac{h_i(A)}{|a_{ii}|}}.\] Hence,
\[\max\{ \mu,1\}
\max\limits_{i\in N} \frac{z_i(A)}{|a_{ii}|}\max\left\{\frac{1}{\mu
- \frac{h_1(A)}{|a_{11}|}}, \frac{1}{ 1-\max\limits_{i\neq
1}\frac{h_i(A)}{|a_{ii}|}}\right\}  =\frac{\max\limits_{i\in
N}\frac{z_i(A)}{|a_{ii}|}}{ 1-\max\limits_{i\in
N}\frac{h_i(A)}{|a_{ii}|}}.\] The proof is completed. \end{proof}

\subsection{\textbf{the optimal value of  $\mu$ for the bound (\ref{m10})}}
First, we consider the Nekrasov matrix $A=[a_{ij}] \in C^{n,n}$ with
\[ \label{eq2.40}|a_{11}| - h_1(A) < \min\limits_{i\neq
1}(|a_{ii}|-h_i(A)),\] and give the following lemmas.

\begin{lem}\label{le2.000}
Let $a, b$ and $c$ be positive real numbers, and $0<a-b<c$. Then
\[ \frac{b+c}{a}<\frac{c}{a-b}.\]
\end{lem}

\begin{proof} we only prove that $ \frac{c}{a-b}-\frac{b+c}{a}>0.$
In fact,
\begin{eqnarray*}  \frac{c}{a-b}-\frac{b+c}{a}&=&
\frac{ac-(a-b)(b+c)}{a(a-b)}\\
&=&\frac{ac-(ab+ac-b^2-bc)}{a(a-b)}\\
&=&\frac{-ab+b^2+bc}{a(a-b)}\\
&=&\frac{b(c-(a-b))}{a(a-b)}>0.
\end{eqnarray*}
The proof is completed. \end{proof}

\begin{lem}\label{le 2.40}
Let $A=[a_{ij}] \in C^{n,n}$ be a Nekrasov matrix with
\[|a_{11}| - h_1(A) < \min\limits_{i\neq
1}(|a_{ii}|-h_i(A)).\] Then
\begin{equation} \label{eq2.60} 1<\frac{\min\limits_{i\neq
1}(|a_{ii}|-h_i(A))+ h_1(A) }{|a_{11}|}<\frac{\min\limits_{i\neq
1}(|a_{ii}|-h_i(A)) }{|a_{11}|- h_1(A)}.\end{equation}
\end{lem}
\begin{proof}
Since $A$ is a Nekrasov matrix, we have $|a_{11}| - h_1(A)>0$,
consequently,  the first Inequality in (\ref{eq2.60}) holds.
Moreover, Let $a=|a_{11}|$, $ b=h_1(A)$ and $ c=\min\limits_{i\neq
1}(|a_{ii}|-h_i(A))$. Then from Lemma \ref{le2.000}, the second
holds.
\end{proof}

We now give an interval of $\mu$ such that the bound (\ref{m10}) of
Theorem \ref{th 2.1} is less than the bound (\ref{eq1.10}) of
Theorem \ref{th1.2}.

\begin{lem}\label{le 2.50}
Let $A=[a_{ij}] \in C^{n,n}$ be a Nekrasov matrix with
\[|a_{11}| - h_1(A) < \min\limits_{i\neq
1}(|a_{ii}|-h_i(A)).\] Then for each $\mu \in \left(1,
\frac{\min\limits_{i\neq 1}(|a_{ii}|-h_i(A)) }{|a_{11}|-
h_1(A)}\right)$,
\begin{eqnarray} \label{eq2.5}||A^{-1}||_\infty &\leq&\frac{ \max\{ \mu,1\} \max\limits_{i\in n}
z_i(A) }{\min\left\{\mu|a_{11}| - h_1(A), \min\limits_{i\neq
1}(|a_{ii}|-h_i(A))\right\}}\nonumber
\\& <&\frac{\max\limits_{i\in
N}z_i(A)} { \min\limits_{i\in N}(|a_{ii}|-
h_i(A))}.\nonumber\end{eqnarray}
\end{lem}

\begin{proof} From Lemma \ref{le 2.40}, we have
\[\mu \in \left(1,\frac{\min\limits_{i\neq
1}(|a_{ii}|-h_i(A))+ h_1(A) }{|a_{11}|}\right]\bigcup
\left[\frac{\min\limits_{i\neq 1}(|a_{ii}|-h_i(A))+ h_1(A)
}{|a_{11}|}, \frac{\min\limits_{i\neq 1}(|a_{ii}|-h_i(A))
}{|a_{11}|- h_1(A)}\right).\] and $ \max\{ \mu,1\}=\mu$.

(I) For $ \mu \in \left(1,\frac{\min\limits_{i\neq
1}(|a_{ii}|-h_i(A))+ h_1(A) }{|a_{11}|}\right],$ then \[
\mu|a_{11}|\leq \min\limits_{i\neq 1}(|a_{ii}|-h_i(A))+ h_1(A) ,\]
that is,
\[\mu|a_{11}| - h_1(A) \leq \min\limits_{i\neq 1}(|a_{ii}|-h_i(A)) .\]
Therefore, \[ \frac{ \max\{ \mu,1\} }{\min\left\{\mu|a_{11}| -
h_1(A), \min\limits_{i\neq 1}(|a_{ii}|-h_i(A))\right\}}=
\frac{\mu}{\mu|a_{11}|-h_1(A)}.\] Consider the function
$f(x)=\frac{x}{|a_{11}|x-h_1(A)}, ~x\in
 \left[1,\frac{\min\limits_{i\neq
1}(|a_{ii}|-h_i(A))+ h_1(A) }{|a_{11}|}\right]$. It is easy  to
prove that $f(x)$ is a monotonically decreasing function of $x$.
Hence, for any $ \mu \in
 \left(1,\frac{\min\limits_{i\neq
1}(|a_{ii}|-h_i(A))+ h_1(A) }{|a_{11}|}\right]$,
\[f(\mu)< f(1),\]
i.e., \[ \frac{\mu}{\mu|a_{11}|-h_1(A)}< \frac{1}{ |a_{11}|
-h_1(A)}= \frac{1}{\min\limits_{i\in N}(|a_{ii}|-h_i(A))},\] which
implies that
\[ \frac{\mu \max\limits_{i\in N}z_i(A)}{\mu|a_{11}|-h_1(A)}<
\frac{ \max\limits_{i\in N}z_i(A)}{\min\limits_{i\in
N}(|a_{ii}|-h_i(A))}.\]
 Hence,
 \[\frac{ \max\{ \mu,1\} \max\limits_{i\in n}
z_i(A) }{\min\left\{\mu|a_{11}| - h_1(A), \min\limits_{i\neq
1}(|a_{ii}|-h_i(A))\right\}}  <\frac{\max\limits_{i\in N}z_i(A)} {
\min\limits_{i\in N}(|a_{ii}|- h_i(A))}.\]

(II) For $ \mu \in \left[\frac{\min\limits_{i\neq
1}(|a_{ii}|-h_i(A))+ h_1(A) }{|a_{11}|}, \frac{\min\limits_{i\neq
1}(|a_{ii}|-h_i(A)) }{|a_{11}|- h_1(A)}\right),$ then
\[ \mu|a_{11}|\geq \min\limits_{i\neq
1}(|a_{ii}|-h_i(A))+h_1(A),\] that is, \[ \mu|a_{11}|-h_1(A)\geq
\min\limits_{i\neq 1}(|a_{ii}|-h_i(A)).\] Therefore, \[ \frac{
\max\{ \mu,1\} }{\min\left\{\mu|a_{11}| - h_1(A), \min\limits_{i\neq
1}(|a_{ii}|-h_i(A))\right\}}= \frac{\mu}{\min\limits_{i\neq
1}(|a_{ii}|-h_i(A))}.\] Consider the function
$g(x)=\frac{x}{\min\limits_{i\neq 1}(|a_{ii}|-h_i(A))}, ~x\in
\left[\frac{\min\limits_{i\neq 1}(|a_{ii}|-h_i(A))+ h_1(A)
}{|a_{11}|}, \frac{\min\limits_{i\neq 1}(|a_{ii}|-h_i(A))
}{|a_{11}|- h_1(A)}\right]$. Obviously, $g(x)$ is a monotonically
increasing function of $x$. Hence, for any $ \mu \in
\left[\frac{\min\limits_{i\neq 1}(|a_{ii}|-h_i(A))+ h_1(A)
}{|a_{11}|}, \frac{\min\limits_{i\neq 1}(|a_{ii}|-h_i(A))
}{|a_{11}|- h_1(A)}\right)$,
\[g(\mu)< g\left(  \frac{\min\limits_{i\neq 1}(|a_{ii}|-h_i(A))
}{|a_{11}|- h_1(A)}\right),\] that is,
\[ \frac{\mu}{\min\limits_{i\neq
1}(|a_{ii}|-h_i(A))} <\frac{1}{ |a_{11}| -h_1(A)}=
\frac{1}{\min\limits_{i\in N}(|a_{ii}|-h_i(A))}, \] which implies
that
\[ \frac{\mu \max\limits_{i\in
N}z_i(A)}{\min\limits_{i\neq 1}(|a_{ii}|-h_i(A))} <\frac{
\max\limits_{i\in N}z_i(A)}{\min\limits_{i\in
N}(|a_{ii}|-h_i(A))}.\]Hence,
 \[\frac{ \max\{ \mu,1\} \max\limits_{i\in n}
z_i(A) }{\min\left\{\mu|a_{11}| - h_1(A), \min\limits_{i\neq
1}(|a_{ii}|-h_i(A))\right\}}  <\frac{\max\limits_{i\in N}z_i(A)} {
\min\limits_{i\in N}(|a_{ii}|- h_i(A))}.\] The conclusion follows
from (I) and (II).\end{proof}

Similar to the proof of Theorem \ref{th2.2},  we can easily
determine the optimal value of $\mu$ by Lemma \ref{le 2.50}.

\begin{thm}\label{th2.20}
Let $A=[a_{ij}] \in C^{n,n}$ be a Nekrasov matrix with
\[|a_{11}| - h_1(A) < \min\limits_{i\neq
1}(|a_{ii}|-h_i(A)).\] Then
\begin{eqnarray} \label{eq2.70} &\min\left\{ \frac{ \max\{ \mu,1\} }{\min\left\{\mu|a_{11}| - h_1(A), \min\limits_{i\neq
1}(|a_{ii}|-h_i(A))\right\}}: \mu \in \left(1,
\frac{\min\limits_{i\neq 1}(|a_{ii}|-h_i(A)) }{|a_{11}|-
h_1(A)}\right)\right\}\nonumber&
\\&= \frac{\min\limits_{i\neq 1}(|a_{ii}|-h_i(A))+ h_1(A)
}{|a_{11}|\min\limits_{i\neq 1}(|a_{ii}|-h_i(A))}.&\end{eqnarray}
Furthermore,
\begin{equation} \label{eq2.80}||A^{-1}||_\infty \leq  \frac{ \max\limits_{i\in N}z_i(A) \left(\min\limits_{i\neq 1}(|a_{ii}|-h_i(A))+ h_1(A)
\right)}{|a_{11}|\min\limits_{i\neq 1}(|a_{ii}|-h_i(A))}
<\frac{\max\limits_{i\in N}z_i(A)} { \min\limits_{i\in N}(|a_{ii}|-
h_i(A))}.\end{equation}
\end{thm}

\begin{rmk} Theorem \ref{th2.20} provides a method to determine the optimal value of $\mu$
for a Nekrasov matrix $A=[a_{ij}] \in C^{n,n}$ with
\[|a_{11}| - h_1(A) < \min\limits_{i\neq
1}(|a_{ii}|-h_i(A)).\]  Also consider the matrix $A_1$. By
computation, we get
\[ |a_{11}| - h_1(A)=  3.8000< 5.2641= \min\limits_{i\neq
1}(|a_{ii}|-h_i(A)).\] Hence, by Theorem \ref{th2.20}, we can obtain
that the bound (\ref{m10}) in Theorem \ref{th 2.1} reaches its
minimum\[ \frac{\max\limits_{i\in N}z_i(A)\min\limits_{i\neq
1}(|a_{ii}|-h_i(A))+ h_1(A) }{|a_{11}|\min\limits_{i\neq
1}(|a_{ii}|-h_i(A))} =  0.4594\] at $\mu=\frac{\min\limits_{i\neq
1}(|a_{ii}|-h_i(A))+ h_1(A) }{|a_{11}|}= 1.2092$ (also see Figure
2).
\end{rmk}

Next, we study the bound (\ref{m10}) in Theorem \ref{th 2.1} for the
Nekrasov matrix $A=[a_{ij}] \in C^{n,n}$ with
\[|a_{11}| - h_1(A) \geq \min\limits_{i\neq
1}(|a_{ii}|-h_i(A)).\]

\begin{thm} \label{th2.30}
Let $A=[a_{ij}] \in C^{n,n}$ be a Nekrasov matrix with
\[|a_{11}| - h_1(A) \geq \min\limits_{i\neq
1}(|a_{ii}|-h_i(A)).\] Then we can take $\mu
=\frac{\min\limits_{i\neq 1}(|a_{ii}|-h_i(A))+ h_1(A) }{|a_{11}|}$
such that
\begin{eqnarray} ||A^{-1}||_\infty &\leq&\frac{ \max\{ \mu,1\} \max\limits_{i\in n}
z_i(A) }{\min\left\{\mu|a_{11}| - h_1(A), \min\limits_{i\neq
1}(|a_{ii}|-h_i(A))\right\}}\nonumber
\\& =&\frac{\max\limits_{i\in
N}z_i(A)} { \min\limits_{i\in N}(|a_{ii}|-
h_i(A))}.\nonumber\end{eqnarray}
\end{thm}

\begin{proof} since $|a_{11}| - h_1(A) \geq \min\limits_{i\neq
1}(|a_{ii}|-h_i(A))$, we have \[ \mu =\frac{\min\limits_{i\neq
1}(|a_{ii}|-h_i(A))+ h_1(A) }{|a_{11}|}\leq 1,\] $ \max\{ \mu,
~1\}=1$, and
\[\frac{ \max\{ \mu,1\} }{\min\left\{\mu|a_{11}| - h_1(A), \min\limits_{i\neq
1}(|a_{ii}|-h_i(A))\right\}}= \frac{1 }{ \min\limits_{i\in
N}(|a_{ii}|- h_i(A))}.\] Hence,
\[\frac{ \max\{ \mu,1\} \max\limits_{i\in n}
z_i(A) }{\min\left\{\mu|a_{11}| - h_1(A), \min\limits_{i\neq
1}(|a_{ii}|-h_i(A))\right\}}  =\frac{\max\limits_{i\in N}z_i(A)} {
\min\limits_{i\in N}(|a_{ii}|- h_i(A))}. \] The proof is completed.
\end{proof}

\begin{rmk} (I) Theorems \ref{th2.2} and \ref{th2.3} provide the value of $\mu$, i.e.,
\[\mu=1+\frac{h_1(A)}{|a_{11}|}-\max\limits_{i\neq 1}
\frac{h_i(A)}{|a_{ii}|}\] such that the bound (\ref{m1}) in Theorem
\ref{th 2.1} is not worse than the bound (\ref{eq1.1}) in theorem
\ref{th1.2} for a Nekrasov matrix $A=[a_{ij}] \in C^{n,n}$. In
particular, for the Nekrasov matrix $A$ with
$\frac{h_1(A)}{|a_{11}|}> \max\limits_{i\neq 1}
\frac{h_i(A)}{|a_{ii}|}$, the bound (\ref{m1}) is better than the
bound (\ref{eq1.1}).

(II) Theorems \ref{th2.20} and \ref{th2.30} provide the value of
$\mu$, i.e.,
\[\mu=\frac{\min\limits_{i\neq 1}(|a_{ii}|-h_i(A))+ h_1(A)
}{|a_{11}|}\] such that the bound (\ref{m10}) in Theorem \ref{th
2.1} is not worse than the bound (\ref{eq1.10}) in theorem
\ref{th1.2} for a Nekrasov matrix $A=[a_{ij}] \in C^{n,n}$. In
particular, for the Nekrasov matrix $A$ with $|a_{11}| - h_1(A) <
\min\limits_{i\neq 1}(|a_{ii}|-h_i(A))$, the bound (\ref{m10})  is
better than the bound (\ref{eq1.10}).
\end{rmk}

\section{Numerical Examples}

\begin{example} Consider the following five Nekrasov matrices in \cite{Cv3}:
\[A_2 =  \left[ \begin{array}{cccc}
  8  &1   &-0.2 &3.3 \\
 7   &13    &2 &-3   \\
 -1.3   &6.7    &13 &-2\\
 0.5 & 3 &1  & 6
\end{array} \right],~ A_3=\left[\begin{array}{cccc}
     21 &  -9.1 &-4.2&-2.1  \\
      -0.7& 9.1 &-4.2&-2.1   \\
      -0.7&   -0.7& 4.9&-2.1  \\
      -0.7&-0.7&-0.7&2.8
\end{array}
\right],\]
\[A_4 =  \left[ \begin{array}{cccc}
  5  &1   &0.2 &2 \\
 1   &21    &1 &-3   \\
 2   &0.5    &6.4 &-2\\
 0.5 & -1 &1  & 9
\end{array} \right],~ A_5=\left[\begin{array}{ccc}
     6 &  -3 &-2  \\
      -1& 11 &-8   \\
      -7&   -3&10
\end{array}
\right],\]
\[A_6=  \left[ \begin{array}{cccc}
  8  &-0.5   &-0.5 &-0.5 \\
 -9   &16    &-5 &-5   \\
 -6   &-4    &15 &-3\\
 -4.9 & -0.9 &-0.9  & 6
\end{array} \right].\]
Obviously, $A_2$, $A_3$ and $A_4$ are SDD. And it is not difficult
to verify that $A_4,~A_5$ satisfy the conditions in Theorems
\ref{th2.2} and \ref{th2.20} and $ A_2,~A_3,~A_6$ satisfy the
conditions in Theorems \ref{th2.3} and \ref{th2.30}.  We compute by
Matlab 7.0 the upper bounds for the infinity norm of the inverse of
$A_i$, $i=2,\ldots,6$, which are showed in Table 1. It is easy to
see  from Table 1 that this example illustrates Theorems
\ref{th2.2}, \ref{th2.3}, \ref{th2.20} and \ref{th2.30},.

\hspace{0.2cm} \noindent
\begin{tabular}[htbp]{cccccc}\hline
 Matrix                               & $A_2$    &$A_3$   &$A_4$   & $A_5$    &$A_6$ \\\hline
 Exact $||A^{-1}||_\infty$            &0.2390   &0.8759   &0.2707  & 1.1519   &0.4474      \\\hline
 Varah                                &1        &1.4286   &0.5556  & --       &--\\\hline
 The bound (\ref{eq1.1})              &0.8848   &1.8076   &0.6200  & 1.4909   &1.1557 \\\hline
 Theorems \ref{th2.2} or \ref{th2.3}  &0.8848   &1.8076   &0.5270  & 1.4266   &1.1557       \\\hline
 The bound (\ref{eq1.10})             &0.6885   &0.9676   &0.7937  & 2.4848   &0.5702 \\\hline
 Theorems \ref{th2.20} or \ref{th2.30}&0.6885   &0.9676   &0.5895  & 1.5923   &0.5702 \\\hline
\end{tabular}

\hspace{0.3cm} \noindent Table 1. The upper bounds for
$||A_{i}^{-1}||_\infty$, $i=2,\ldots,6$.
\end{example}


\section*{Acknowledgements}
This work is supported by National Natural Science Foundations of
China (11361074, 11326242) and Natural Science Foundations of Yunnan
Province (2013FD002).







\end{document}